\documentclass[a4paper, 11pt]{article}
\usepackage[latin1]{inputenc}    
\usepackage[T1]{fontenc}
\usepackage{lmodern}
\usepackage{graphicx}
\usepackage[french,english]{babel}
\usepackage{subfigure}
\usepackage{float} 
\usepackage{fullpage}
\usepackage{amsmath, amsfonts, amssymb} 
\usepackage{hyperref} 
\usepackage[french,ruled,vlined]{algorithm2e}
\usepackage{dsfont}
\usepackage{amssymb}
\usepackage{amsthm}
\hypersetup{colorlinks = true, urlcolor = blue, bookmarksopen = true}

\def\ra{{\rightarrow}}

 \title{Support convergence for  the spectrum of Wishart matrices with correlated entries }
 \author{Alice Guionnet$^\dagger$ and Kevin Richard$^\flat$} 
   \date{02/2014} 

 \newtheorem{thm}{Theorem}
\newtheorem{lemma}[thm]{Lemma}
\newtheorem{remark}[thm]{Remark}
\newtheorem{corollaire}[thm]{Corollary}

\def\Tr{{\rm Tr}}
\newtheorem{hypothesis}[thm]{Hypothesis}

\begin{document}

\maketitle

 \thanks{$\dagger$: MIT, Mathematics department, 77 Massachusetts Av, Cambridge MA 02139-4307, USA, guionnet@math.mit.edu,   and CNRS \& \'Ecole Normale Sup\'erieure
de Lyon, Unit\'e de math\'ematiques pures et appliqu\'ees, 46 all\'ee
d'Italie, 69364 Lyon Cedex 07, France.\\
Research partially supported by  Simons foundation and NSF award DMS-1307704.}

\thanks{$\flat$: ENS Cachan, Mathematics department, 61 avenue du Pr\'esident Wilson, 94230 Cachan, France, krichard@ens-cachan.fr.}

\section{Introduction}
We consider Wishart matrices given by
\begin{equation}\label{Wis}Z_{n,m}=\frac{1}{n} \sum_{1\le p\le m} X_pX_p^*\end{equation}
where $X_p, p\ge 0$ are i.i.d $n$-dimensional vectors.  These matrices were intensively studied in connection with statistics. When the entries of $X_p$ are independent, equidistributed and with finite second  moment, 
it was shown by Marchenko and Pastur \cite{MP} that the empirical measure $L_n  = n^{-1} \sum_{i=1}^{n} \delta_{\lambda_i}$ of the eigenvalues of such matrices converges almost surely as $n,m$ go to infinity so that $m/n$ goes to $c\in (0,+\infty)$. This result was extended 
by A.Pajor and L.Pastur \cite{PP}  in the case where the entries of the $X_p$'s are correlated but
have a log-concave isotropic law. Very recently, these authors together with O. Gu\'edon and A. Lytova, proved the central limit theorem for the centered linear statistics \cite{GLPP} in the setting of \cite{PP}. To that end, they additionally assume that the law of the entries are ``very good'', see \cite[Definition1.6]{GLPP}, in the sense that mixed moments of degree four satisfy asymptotic conditions and quadratic forms satisfy concentration of measure property. 

In this article we will consider also the case where 
the entries of the $X_p$'s are correlated but
have a strictly log-concave law.  We will show, under some symmetry and convergence hypotheses, that 
 the
central limit theorem for linear statistics    holds around their limit, and deduce the  convergence of the eigenvalues  to the support of the limiting measure. To prove this result we shall assume that the law of the entries is ``very good'' in the sense of \cite{GLPP}, but in fact even more that it is symmetric and with strictly log-concave law. The two later assumptions could possibly be removed.

The fluctuations of the spectral measure   around the limiting measure or around the expectation were first studied by Jonsson \cite{jonsson} then by Pastur \emph{et al.} in \cite{KKP96},   and Sinai and Soshnikov \cite{sinai} with
$p\ll N^{1/2}$ possibly  going to infinity with $N$.  Since then, a long list of further-reaching results   have been obtained:  the central limit theorem was extended to the so-called matrix models where the entries interact via a potential in  \cite{johansson},  
   the set of test functions was extended and the assumptions on the entries of the Wigner matrices weakened in \cite{BaiYaoBernoulli2005,BAI2009EJP,lytova,MShcherbina11}, 
   Chatterjee developed a general approach to these questions in \cite{chatterjee}, under the condition that the law $\mu$ can be written as a transport of the Gaussian law... Here, we will follow mainly the approach developed by Bai and Silverstein in \cite{BS} to
   study the fluctuations of linear statistics in the case where the vectors may have dependent entries. Since the fluctuations of the centered linear statistics were already studied in \cite{GLPP}, 
   we shall concentrate on the convergence of the mean of linear statistics (even though Bai and Silverstein method extends to obtain this result as well).
   We show that the mean, as the covariance (see \cite{GLPP})  will depend on several fourth joint moments of the entries of this vector. 

Convergence of the support of the eigenvalues towards the support of the limiting measure  was shown in \cite{BY} and \cite{BS0} in the case of independent entries.  We show that this convergence still
holds  in the case of dependent entries with log-concave distribution and
therefore that such a dependency can  not result in  outliers.  

\subsection{Statement of the results}
We consider a random matrix $Z_{n,m}$ given by \eqref{Wis},  with  $m$ independent copies of a $n$-dimensional vector $X$  whose entries maybe correlated. More precisely we assume that $X$ follows 
the following distribution: 

$$d\mathbb P(X) = \frac{1}{Z_n} e^{-V(X_1,\dotsc,X_n)}\prod_{i=1}^n dX_i.$$

\begin{hypothesis}\label{hypo1}
We will assume that the mean of $X$ is zero, that its covariance matrix is the identity matrix and that the four moments are homogeneous : $$\forall i \in [\![1,n]\!], \mathbb{E}[X_i^4] = \mu.$$

Besides, the following condition holds for the Hessian matrix of $V$ : \begin{equation}\label{hess}\mathrm{Hess}(V)  \ge \frac{1}{C} I_n, \: C > 0.\end{equation}
Moreover, the law of $X$ is symmetric, that is $\mbox{Law}(X_{\sigma(1)},\ldots, X_{\sigma(n)})=\mbox{Law}(X_1,\ldots, X_n)$ for all permutation $\sigma$ of $\{1,\ldots,n\}$.

\end{hypothesis}

 This implies
that  concentration inequalities hold for the vector $X$  (see the Appendix). In particular, we have $\mathrm{Var}(\sum X_i^2) = O(n)$. Thus, it is natural to assume  the  following condition : 
\begin{hypothesis}\label{hypo2} There exists $\kappa>0$ so that 
$$ \lim_{n \to \infty} n^{-1} \mathrm{Var}(\sum X_i^2) = \kappa.$$ 
\end{hypothesis}
An example of potential $V$ fulfilling both hypotheses \ref{hypo1} and \ref{hypo2} is given in section \ref{subexample}.

Let us consider the matrix $Y_{n,m}$ whose columns are $m$ independent copies of the vector $\frac{1}{\sqrt{n}}X$. We denote by $W_{n,m}$ the symmetric block matrix of size $n+m$ : $$W_{n,m} \: = \: \begin{pmatrix} 0 & {Y_{n,m}}^* \\ Y_{n,m} & 0 \end{pmatrix}\,.$$

Let $F_{n,m}$ be the empirical spectral distribution of $W_{n,m}$ defined by : $$F_{n,m}(dx) \: = \: \frac{1}{n+m} \sum_{i=1}^{n+m} \delta_{\lambda_i}.$$
Note that
$$\frac{1}{n}{\rm Tr}( f(Z_{n,m}))=\frac{n+m}{2n} \int f(x^2) F_{n,m}(dx)+\frac{m-n}{2n} f(0)$$
so that the study of the eigenvalues of $W_{n,m}$ or $Z_{n,m}$ are equivalent.
We will assume that 
\begin{hypothesis}\label{hypo3}
$(m/n) \to c \in [1, \infty)$.  Moreover,  $cn-m $ converges towards $ \sigma$ as $n,m$ go to infinity.\end{hypothesis}

It has been proved by A. Pajor and L.Pastur that, when $n \to \infty, m \to \infty$ and $(m/n) \to c \in [1, \infty)$,  the spectral measure $L_n$ of $Z_{n,m}$ converges almost surely to the Marchenko-Pastur law. It is then not difficult
to prove that
 $F_{n,m}$ converges almost surely  to a probability measure $F$. In this paper, we will study the weak convergence of the centered  measure $M_{n,m}(dx) \: = (n+m)(\mathbb E[F_{n+m}](dx)-F(dx))$.

Let  $A_K$ be the set of the set of $C^K$   functions $f$.  We will focus on the empirical process $M_{n,m} := {M_{n,m}(f)}$ indexed by $A_K$ : $$ M_{n,m}(f) \: := \: (n+m) \int_{\mathbb{R}} f(x)(\mathbb E[F_{n,m}] - F)(dx),  \: \: \qquad f \in A_K.$$

\begin{thm}\label{THMclt}
Under hypotheses  \ref{hypo1}, \ref{hypo2},  and \ref{hypo3}, there exists $K<\infty$
so that  the process $M_{n,m} := (M_{n,m}(f))$ indexed by the set $A_K$  converges to a process $M:= \{ M(f), f \in A_K \}$. Moreover,  this process depends on both $\kappa$ and $\mu$.
\end{thm}

As a non trivial corollary, we show that the support of the matrix $W_{n,m}$ converges towards the support of  the law $F$, namely our main theorem:

\begin{thm}\label{THMeig}
Under hypotheses \ref{hypo1}, \ref{hypo2}, and  \ref{hypo3}, the support of the eigenvalues 
of $ W_{n,m}$ converges almost surely towards $[-1-\sqrt{c}, 1-\sqrt{c}]\cup [\sqrt{c}-1,\sqrt{c}+1]$.
\end{thm}
Theorem \ref{THMclt} can be coupled with the central limit theorem for the centered statistics derived in \cite{GLPP} to derive the convergence of the 
random process $$ G_{n,m}(f) \: := \: (n+m) \int_{\mathbb{R}} f(x)(F_{n,m} - F)(dx),  \: \: \qquad f \in A_K.$$
The covariance of the limiting Gaussian process also depends on $\kappa$ and $\mu$, see \cite{GLPP}.

\subsection{Strategy of the proof}\label{sec12}
 To prove  Theorems \ref{THMeig}  and  \ref{THMclt}, we shall prove that Theorem \ref{THMclt} holds when $f$ is taken in
 the set $(z-x)^{-1}, z\in \mathbb C$. We then use a now standard strategy to generalize it to smooth enough functions that we describe below.
  More precisely,  we denote by $s_H$ the Stieltjes transform for the measure $H$ defined by : $$s_H(z) \: = \: \int_{\mathbb{R}} \frac{dH(x)}{x-z},  \: \: \qquad z \notin supp(H), $$
we let $s_{n,m}(z)$ and $s(z)$ be the Stieltjes transform of $F_{n,m}$ and $F$ respectively. We set $M_{n,m}$ to be the process
   $$M_{n,m}(z) \: = \: (n+m)(\mathbb E[ s_{n,m}(z)] - s)(z)$$
   indexed by $z$, $Im z \neq 0$.
We denote by $\mathbb{C}_{v_0}$ the set $\mathbb{C}_{v_0} \: = \: \{z = u + iv, |v| \ge v_0\,, |z| < R\}$, where R is an arbitrary  big enough constant. Then we shall prove: 

\begin{thm}\label{THMformulaclt} Assume Hypotheses \ref{hypo1}, \ref{hypo2}  and \ref{hypo3} hold. Then, there exists a constant $c>0$ such that 
for $v_0\ge n^{-c}$  the process $\{ M_{n,m}(z) ;  z\in \mathbb{C}_{v_0} \}$ converges uniformly   to the  process $\{ M(z) ; z\in \mathbb{C}_{v_0} \}$ 
 \begin{align*} M(z) := & \sigma s^2(z) + (1 + c \partial_z s^2(z)) (s^1(z))^3 \left[  - \frac{\partial_z s^1(z)}{ (s^1(z))^2} + c \left[ (4\mu -5)(s^2(z))^2 + 2\partial_z s^2(z)\right] \right] \\ & +  (1 + \partial_z s^1(z)) c (s^2(z))^3 \left[ - \frac{\partial_z s^2(z)}{(s^2(z))^2} +  \left[ 2(\mu - 2 + \frac{\kappa}{2})(s^1(z))^2 + 2 \partial_z s^1(z)\right] \right] \end{align*} where \begin{align*} & s^1(z) = - \frac{1}{2 z} \left(z^2-c+1 - \sqrt{(z^2-c+1)^2-4z^2} \right)\\ & s^2(z) = - \frac{1}{2 c z} \left(z^2+c-1 - \sqrt{(z^2-1+c)^2-4cz^2} \right)\,. \end{align*}
\end{thm}
To deduce  Theorems \ref{THMeig}  and  \ref{THMclt} from the above result, we rely on the 
  following  expression of  \cite[ formula (5.5.11)]{AGZ}, which  allows to reconstruct 
integrals with respect to a measure from its Stieltjes transform. Namely, if  $f$ is a $C^K$ compactly supported  function,  and if we set 
 $$\Psi_f(x,y) = \sum_{l=0}^K \frac{i^l}{l!} f^{(l)}(x) g(y)y^l  $$
 with $g : \mathbb{R} \mapsto [0,1]$, $g$ a smooth function with compact support, $g = 1$ near the origin and $0$ outside $[-c_0,c_0]$, with $c_0$ an arbitrary constant, then for any probability measure $\mu$ on the real line
 $$\int f (t)d\mu(t)=\Re \int_0^\infty dy\int_{-\infty}^\infty dx \left(\int\frac{\bar\partial \Psi(x,y)}{t-x-iy} \mu(dt)\right)\,.$$
 Here we have denoted $\bar\partial=\pi^{-1}(\partial_x+i\partial_y)$.
 Hence, 
 we have 
 \begin{eqnarray*}
 M_{n,m}(f) &=&\mathbb E[ \Re \int_0^{+ \infty} dy \int_{- \infty}^{+ \infty} dx \int \frac{\bar{\partial} \Psi_f(x,y)}{t-x-iy} (n+m) [F_{n,m} - F](t)dt]\\
 &=&  \Re \int_0^{+ \infty} dy \int_{- \infty}^{+\infty} dx {\bar{\partial} \Psi_f (x,y)}M_{n,m}(x+iy) \end{eqnarray*}
Theorem \ref{THMformulaclt} implies the convergence of the above integral for $y\ge n^{-c}$, with $M_{n,m}$ replaced by its limit $M$ [note here that $\Psi_f$ is compactly supported]. On the other hand, 
 $ |{\bar{\partial} \Psi_f (x,y)}|/y^K$ and $n^{-1}y M_{n,m} (x+iy)$ are uniformly bounded, so that if $Kc>1$, the integral over $[0,n^{-c}]$ 
 is neglectable. Hence, we conclude that for such functions $f$, we have 
 \begin{equation}\label{tuy}\lim_{n\ra\infty}  M_{n,m}(f)= \Re \int_0^{+ \infty} dy \int_{- \infty}^{+\infty} dx {\bar{\partial} \Psi_f (x,y)}M(x+iy) \,.\end{equation}
 This convergence extends to non-compactly supported $C^K$
functions by the rough estimates on the eigenvalues derived in Lemma \ref{borne}. This completes the proof of Theorem \ref{THMclt}.

To deduce from \eqref{tuy}   the convergence of the support of the empirical measure, that is Theorem \ref{THMeig}, we finally take $f$ $C^K$ and 
vanishing  on the support of the Pastur-Marchenko distribution. But then $\Psi_f$ also vanishes when $x$ belongs to the support
of the Pastur-Marchenko distribution. Since  $M(x+iy)$ is analytic away from this support (as it is a smooth function of $s^1$ and $s^2$ 
which are analytic there), ${\bar{\partial}M(x+iy) }$ vanishes 
on the support of integration.  This shows after an integration by parts that 
\begin{eqnarray*}
\lim_{n\ra\infty}  M_{n,m}(f)&=& \Re \int_0^{+ \infty} dy \int_{- \infty}^{+\infty} dx {\bar{\partial} [\Psi_f (x,y)}M(x+iy)]\\
&=&\Re[i\pi^{-1}  \int_0^{+ \infty} dy \int_{- \infty}^{+\infty} dx {\partial_y [\Psi_f (x,y)}M(x+iy)]\end{eqnarray*}
where we noticd that the integral of $\partial_x [\Psi_f(x,y) M(x+iy)]$ vanishes as $\Psi_f(x,y)$ vanishes when $x$ is outside a compact. On the other hand
$$ \int_0^{+ \infty} dy \int_{- \infty}^{+\infty} dx {\partial_y [\Psi_f (x,y)}M(x+iy)]=i\pi^{-1}\int f(x) M(x) dx$$
is purely imaginary. Hence, we conclude that  for any $f$ compactly supported, $C^K$ and 
vanishing  on the support of the Pastur-Marchenko distribution, we have 
$$\lim_{n\ra\infty}\mathbb E[\sum f(\lambda_i)]=\lim_{n\ra\infty}  M_{n,m}(f)=0\,.$$
Taking $f$ non-negative and greater than one on some compact which does not intersect the support of the Pastur-Marchenko law shows 
that  the probability that 
there are  eigenvalues  in this compact as $n$ goes to infinity  
vanishes. By Lemma \ref{borne}, we conclude that the probability that there is an eigenvalue at distance $\epsilon$ of the support of the Pastur-Marchenko law goes to zero
as $n$ goes to infinity. 
But we also have concentration of the extreme eigenvalues Theorem \ref{conclargest} and therefore the convergence holds almost surely. 
\section{A few useful results}

In this section, we review a few classical results about the convergence of $s_{n,m}$ and provide some proofs on which we shall elaborate to derive Theorem \ref{THMformulaclt}. In particular, we emphasis on which domain the convergence holds, in preparation to the proof of Theorem
\ref{THMeig}.

To simplify the computations, we are going to introduce a few notations. Let $$S = (W_{n,m} - z)^{-1}\,.$$ We write  $\alpha_k$ the vector obtained from the $k$-th column $W_{n,m}$ by deleting the $k$-th entry and $W_{n,m}(k)$ the matrix resulting from deleting the $k$-th row and column from $W_{n,m}$. Let  $S_k$ be $S_k = (W_{n,m}(k) - z)^{-1}$, we write $S_1$ (respectively $S_2$) the submatrix of order $n$ formed by the last $n$ row and columns (respectively  the submatrix of order $m$ formed by the first $m$ row and columns). Here are some useful notations : 

$$s_{n,m}^1(z) = \frac{1}{n} \mathrm{Tr}(S_1) \qquad s_{n,m}^2(z) = \frac{1}{m} \mathrm{Tr}(S_2)\qquad $$
$$s_{n,m}(z) = \frac{1}{n+m}\mathrm{Tr}(S) = \frac{n s_{n,m}^1(z) + m s_{n,m}^2(z)}{n+m} .$$
$$\beta_k = z + \frac{1}{n} {\alpha_k}^* S_k \alpha_k,$$
$$\epsilon_k = -\frac{1}{n} {\alpha_k}^* S_k \alpha_k + \mathbb{E}[s_{n,m}^1(z)], \: \: \qquad k \le m$$
\begin{equation}\label{defeps}\epsilon_k = -\frac{1}{n} {\alpha_k}^* S_k \alpha_k + c_{n,m} \mathbb{E}[s_{n,m}^2(z)], \: \: \qquad k > m, \: \: c_{n,m} = \frac{m}{n}\end{equation}
$$\delta_1(z) =  -\frac{1}{n} \sum_{k=m+1}^{n+m} \frac{\epsilon_k}{\beta_k(z+ c_{n,m} \mathbb{E}[s_{n,m}^2(z)])},$$
$$\delta_2(z) = - \frac{1}{m} \sum_{k=1}^{m} \frac{\epsilon_k}{\beta_k(z+ \mathbb{E}[s_{n,m}^1(z)])},$$


\begin{remark}Since $W_{n,m}$ is a real symmetric random matrix, the eigenvalues of the matrix $S$ can be written as $\frac{1}{\lambda_j - z}$, with $\lambda_j$ the real eigenvalues of the matrix $W_{n,m}$. Thus, all the eigenvalues are bounded by $\frac{1}{| \Im z |}$. We also have $|\mathrm{Tr}(S)| \le \frac{(n+m)}{| \Im z |}$ and $|\mathrm{Tr}(S^2)| \le \frac{(n+m)}{| \Im z |^2}$. 
Moreover, we have $\frac{dS(z)}{dz} \: = \: S^2(z)$.

\end{remark}

\begin{thm}\label{THMMPlaw}
For $v_0\ge n^{-\frac{1}{13}}$, uniformly on  $z \in \mathbb{C}_{v_0}$, $s_{n,m}(z)$ goes to $s(z)$ in $L_2$, where $$s(z) = -\frac{1}{1+c} \left( z - \frac{\sqrt{(z^2-c+1)^2-4z^2}}{z} \right) = \frac{s^1(z) + c s^2(z)}{1 + c}$$ 
with for $i=1,2$
$$s^i(z)= - \frac{1}{2 c^{i-1}z} \left(z^2+(-1)^i(c-1 ) - \sqrt{(z^2++(-1)^i(c-1 ))^2-4c^{i-1}z^2} \right) \,.$$
\end{thm}

\begin{proof}

By Lemma \ref{Schur} in the appendix, we have $$\mathrm{Tr}(S_1) = - \sum_{k=m+1}^{n+m} \frac{1}{\beta_k}$$ where, for $k > m$, we have denoted
 $$\frac{1}{\beta_k} = \frac{1}{z + c_{n,m} \mathbb{E}[s_{n,m}^2(z)] - \epsilon_k} = \frac{1}{z + c_{n,m} \mathbb{E}[s_{n,m}^2(z)]} + \frac{\epsilon_k}{(z + c_{n,m} \mathbb{E}[s_{n,m}^2(z)] - \epsilon_k)(z + c_{n,m} \mathbb{E}[s_{n,m}^2(z)])}.$$
Thus, $$s_{n,m}^1(z) = - \frac{1}{z+ c_{n,m} \mathbb{E}[s_{n,m}^2(z)]} + \delta_1(z).$$

By the same method, we obtain $$s_{n,m}^2(z) = - \frac{1}{z+ \mathbb{E}[s_{n,m}^1(z)]} + \delta_2(z).$$

We can deduce, by taking the mean in both previous equalities, that $\mathbb{E}[s_{n,m}^1(z)]$ and $\mathbb{E}[s_{n,m}^2(z)]$ are solutions of the following system $ \begin{cases} \mathbb{E}[s_{n,m}^1(z)] =  - \frac{1}{z + c_{n,m} \mathbb{E}[s_{n,m}^2(z)]} + \mathbb{E}[\delta_1(z)] \\ \mathbb{E}[s_{n,m}^2(z)] =  - \frac{1}{z + \mathbb{E}[s_{n,m}^1(z)]} + \mathbb{E}[\delta_2(z)] \end{cases}.$

Thus, $\mathbb{E}[s_{n,m}^1(z)]$ is a solution of the quadratic equation $$a_2 \mathbb{E}[s_{n,m}^1(z)]^2 + a_1 \mathbb{E}[s_{n,m}^1(z)] + a_0 = 0$$ with \begin{align} & a_2 = z + c_{n,m} \mathbb{E}[\delta_2(z)]\nonumber \\ & a_1 = z^2 -c_{n,m} +1 +z c_{n,m} \mathbb{E}[\delta_2(z)]-z \mathbb{E}[\delta_1(z)]-c_{n,m} \mathbb{E}[\delta_1(z)]\mathbb{E}[\delta_2(z)]\label{toto}\\ & a_0 = z - z^2 \mathbb{E}[\delta_1(z)] + c_{n,m} \mathbb{E}[\delta_1(z)] - z c_{n,m} \mathbb{E}[\delta_1(z)]\mathbb{E}[\delta_2(z)]. \nonumber\end{align}
We deduce  that \begin{equation}\label{pok} \mathbb{E}[s_{n,m}^1(z)]  = -\frac{\left(a_1 - \sqrt{a_1^2 - 4 a_0 a_2}\right)}{2 a_2}.\end{equation}


We only need to estimate $a_0$, $a_1$ and $a_2$  to find the limit of $\mathbb{E}[s_{n,m}^1(z)]$. We now  show that $\mathbb{E}[\delta_1(z)] = O(\frac{1}{n})$ and $\mathbb{E}[\delta_2(z)] = O(\frac{1}{n})$.   Using the fact that, for $k \le m$, $$\frac{1}{\beta_k} = \frac{1}{z + \mathbb{E}[s_{n,m}^1(z)] - \epsilon_k} = \frac{1}{z + \mathbb{E}[s_{n,m}^1(z)]} + \frac{\epsilon_k}{(z + \mathbb{E}[s_{n,m}^1(z)] - \epsilon_k)(z + \mathbb{E}[s_{n,m}^1(z)])},$$ we can write $$  \mathbb{E}[\delta_2(z)] =  - \frac{1}{m} \sum_{k=1}^{m}\mathbb{E} \left[ \frac{\epsilon_k}{(z+\mathbb{E}[s_{n,m}^1(z)])^2} + \frac{\epsilon_k^2}{\beta_k (z+ \mathbb{E}[s_{n,m}^1(z)])^2} \right] .$$
On the other hand, we have $| \beta_k | \ge | \Im(\beta_k) | = | \Im(z) | \ge v_0$ and  $| z+\mathbb{E}[s_{n,m}^1(z)] | \ge v_0$,  which yields

\begin{equation}\label{lkj}  | \mathbb{E}[\delta_2(z)]| \le \frac{v_0^{-2} + v_0^{-3}}{m} \sum_{k=1}^{m}( | \mathbb{E}[\epsilon_k] | + \mathbb{E}[\epsilon_k^2])\,.\end{equation}
To estimate the first term note that we have 
\begin{equation}\label{lkj2}
\mathbb{E}[\epsilon_k] = \frac{1}{n}\mathrm{Tr} (S_1 - (S_k)_1)\,.\end{equation}
 To compute this term, we will use the following equality : 
$$ {\begin{pmatrix} A & B \\ C & D \end{pmatrix}}^{-1} = \begin{pmatrix} A^{-1} + A^{-1} B F^{-1} C A^{-1} & -A^{-1} B F^{-1} \\ -F^{-1} C A^{-1} & F^{-1} \end{pmatrix},$$ where  $F=D - C A^{-1} B$.
Since $S_1$ is the $n \times n$ block in the right bottom, if we let $A$,$B$ and $C$ be $A=D=-zI_n$, $B=Y_{n,m}^*$ and $C=Y_{n,m}$, we have $S_1 = F^{-1} = -zI_n + z^{-1}Y_{n,m}Y_{n,m}^*$. Therefore, we have
 $$\mathrm{Tr} (S_1)=  \mathrm{Tr}(z^{-1}Y_{n,m} Y_{n,m}^* - z)^{-1}.$$

Using the notation $Y_{n,m}(k)$ to denote the submatrix of $Y_{n,m}$ obtained by deleting its $k$-th column (its size is $n \times (n-1)$), we find by a similar method $$ \mathrm{Tr}(S_k)_1 =   \mathrm{Tr}(z^{-1} Y_{n,m}(k) Y_{n,m}(k)^* - z)^{-1}.$$ 
Therefore, we deduce that $$\mathrm{Tr} (S_1 - (S_k)_1) =  z \left[ \mathrm{Tr}(Y_{n,m} Y_{n,m}^* - z^2)^{-1} - \mathrm{Tr}(Y_{n,m}(k) Y_{n,m}(k)^* - z^2)^{-1} \right].$$
But,  $Y_{n,m}(k) Y_{n,m}(k)^* = Y_{n,m} Y_{n,m}^* - \frac{1}{n} \alpha_k \alpha_k^*$, so that we  have proved the equality
\begin{equation}\label{lkj3}
\mathrm{Tr} (S_1 - (S_k)_1) = - z \frac{1}{n} \alpha_k^* (Y_{n,m} Y_{n,m}^* - z^2)^{-1}(Y_{n,m}(k) Y_{n,m}(k)^* - z^2)^{-1} \alpha_k.\end{equation}

Since $Y_{n,m} Y_{n,m}^*$ is a non-negative hermitian matrix, its eigenvalues are non-negative. The eigenvalues of $(Y_{n,m} Y_{n,m}^* - z^2)^{-1}$ can be written as $1/(\lambda - z^2)$ or, as we prefer, $1/(\sqrt{\lambda} - z)(\sqrt{\lambda} + z)$. Thus, the absolute value of the eigenvalues of $(Y_{n,m} Y_{n,m}^* - z^2)^{-1}$ are all bounded by $v_0^{-2}$. The same reasoning holds for $(Y_{n,m}(k) Y_{n,m}(k)^* - z^2)^{-1}$. Therefore, we deduce from \eqref{lkj3} that

$$| \mathbb{E}[\mathrm{Tr} (S_1 - (S_k)_1)] | \le |z| v_0^{-4} \frac{1}{n}\mathbb{E}[\|\alpha_k\|^2].$$
But, by Hypothesis \ref{hypo1}, we have $$\mathbb{E}[\|\alpha_k\|^2] = \mathbb{E}[ \sum_{i=1}^n X_i^2] = n.$$
Thus, we have shown according to \eqref{lkj2} that
 $$|\mathbb{E}[ \epsilon_k]| \le \frac{1}{n} R v_0^{-4}$$
Moreover, $$\mathbb{E}[\epsilon_k^2] \le 2 \left( \mathbb{E}\left[ \frac{1}{n^2}\left| \alpha_k^* S_k \alpha_k - \mathrm{Tr}(S_k)_1\right| ^2 \right] +\mathbb{E}\left[ \frac{1}{n^2}\left| \mathrm{Tr}((S_k)_1-S_1) \right| ^2 \right] \right).$$
The concentration of measure  Theorem \ref{concentrationthm} states that the first term is of order  $O(n^{-1})$. For the second one, by the previous computation, we deduce that
 \begin{equation}\label{pol}
\left| \mathrm{Tr}((S_k)_1-S_1) \right| ^2 \le R^2 v_0^{-8} \frac{1}{n^2} \|\alpha_k\|^4.\end{equation}
Since $\mathbb{E}[\frac{1}{n^2} \|\alpha_k\|^4] = O(1)$, we conclude that  $\mathbb{E}[\epsilon_k^2] = O(v_0^{-8}n^{-1})$.

We have now shown according to \eqref{lkj} that \begin{equation}\label{poke}\mathbb{E}[\delta_2(z)] = O(\frac{v_0^{-12}}{n}).\end{equation}
We find  a similar bound for $\mathbb{E}[\delta_1(z)] $.
Since $\mathbb{E}[\delta_i(z)] = O(v_0^{-12}n^{-1})$, $i=1,2$,   a Taylor expansion gives us 

\begin{eqnarray}
\mathbb{E}[s_{n,m}^1(z)] &=& - \frac{1}{2 z} \left(z^2-c+1 - \sqrt{(z^2-c+1)^2-4z^2} \right) + O(\frac{v_0^{-12}}{n}).\label{pol2}\\
\mathbb{E}[s_{n,m}^2(z)] &=& - \frac{1}{2 c z} \left(z^2+c-1 - \sqrt{(z^2-1+c)^2-4cz^2} \right) + O(\frac{v_0^{-12}}{n}).\nonumber\end{eqnarray}

Since $\mathbb{E}[s_{n,m}(z)] = \frac{n \mathbb{E}[s_{n,m}^1(z)] + m  \mathbb{E}[s_{n,m}^2(z)]}{n+m}$, we conclude that  $$\mathbb{E}[s_{n,m}(z)]  = \frac{s^1(z) + c s^2(z)}{1 + c} + O(\frac{v_0^{-12}}{n}).$$
Hence $\mathbb{E}[s_{n,m}(z)] $ converges towards $s(z)$, uniformly on $\mathbb C_{v_0}$ for $v_0\ge n^{-\frac{1}{13}}$. 

Now, for the convergence of $s_{n,m}(z)$ to $s(z)$ in $L_2$, let $f$ be the function $f : x \mapsto (x-z)^{-1}$. $f$ is a Lipschitz function, whose Lipschitz constant is bounded by $v_0^{-2}$. Corollary \ref{concspectral} states that the variance of $\mathrm{Tr}(S)$ is uniformly bounded for all $(n,m)$. From that, we see that $$ \mathbb{E}[|s_{n,m}(z) - \mathbb{E}[s_{n,m}(z)]|^2] = O(v_0^{-4}n^{-2}).$$

Thus, $s_{n,m}(z)$ converges to $s(z)$ in $L_2$.

\end{proof}

\begin{thm}\label{THMspeed} For $v_0\ge n^{-\frac{1}{16}}$, uniformly on 
$ z \in \mathbb{C}_{v_0}$,  $$\mathbb{E}[ {\max_{i \le m} | S_{ii} - s^2(z) |} ^2] = O(v_0^{-16}n^{-1})\quad\mbox{ and }\quad\mathbb{E}[ {\max_{i > m} | S_{ii} - s^1(z) |} ^2]=O(v_0^{-16}n^{-1})\,.$$
\end{thm}

\begin{proof}
Let $i \le m$. By definition, we have \\ \begin{align*} S_{ii} \: &= \: - \frac{1}{z + \frac{1}{n} \alpha_i^* S_i \alpha_i} \\ &= \frac{1}{-z - s^1(z)} + \frac{- s^1(z) +  \frac{1}{n} \alpha_i^* S_i \alpha_i}{(-z - \frac{1}{n} \alpha_i^* S_i \alpha_i)(-z-s^1(z))}. \end{align*}

But, we know that $s^2(z)= (-z - s^1(z))^{-1}$. We thus deduce that $$ S_{ii} - s^2(z) =  \frac{- s^1(z) +  \frac{1}{n} \alpha_i^* S_i \alpha_i}{(-z - \frac{1}{n} \alpha_i^* S_i \alpha_i)(-z-s^1(z))}.$$

Since $z \in \mathbb{C}_{v_0}$, there exist $K > 0$ such that $|(-z - \frac{1}{n} \alpha_i^* S_i \alpha_i)(-z-s^1(z))| > Kv_0^2$. Therefore, we have \\ \begin{align*} | S_{ii} - s^2(z) | &\le K v_0^{-2} |- s^1(z) +  \frac{1}{n} \alpha_i^* S_i \alpha_i | \\ & \le K v_0^{-2} ( \frac{1}{n} | \alpha_i^* S_i \alpha_i - \mathbb{E}[ \mathrm{Tr}(S_i)_1 ] | +  \frac{1}{n} | \mathbb{E}[ \mathrm{Tr}((S_i)_1 - S_1) ] | +  | \mathbb{E}[s_{n,m}^1(z)] - s^1(z) |. \end{align*}

Both last terms converge uniformly to $0$ for all $i$ in $L_2$ (this result has been proved previously) in $O(v_0^{-12}n^{-1})$. For the first one, using the concentration inequality, we can show that $L^2$ norm of the first term is in $O(v_0^{-2}n^{-1})$, where the $O(v_0^{-2}n^{-1})$ is uniform for all $i$ (indeed, the constants are independent of $i$), which concludes the proof for $i \le m$. The proof is similar for $i>m$.

Therefore, for $i > m$, $\mathbb{E}[| S_{ii} - s^1(z) |^2] = O(v_0^{-16}n^{-1})$, where the upper bound is uniform for all $i$.

\end{proof}

\section{Convergence of the process $M_{n,m}(z)$}

We are going to compute the limit of the  function of $M_{n,m}(z)$.

\begin{thm}\label{THM2cltformula} Under the hypotheses \ref{hypo1},\ref{hypo2}  and \ref{hypo3}, uniformly on  $z \in \mathbb{C}_{v_0}$, $v_0\ge n^{-1/20}$,  the function $M_{n,m}(z)$ converges  to the function $M$ defined by \begin{align*} M(z) := & \sigma s^2(z) + (1 + c \partial_z s^2(z)) (s^1(z))^3 \left[  - \frac{\partial_z s^1(z)}{ (s^1(z))^2} + c \left[ (4\mu -5)(s^2(z))^2 + 2\partial_z s^2(z)\right] \right] \\ & +  (1 + \partial_z s^1(z)) c (s^2(z))^3 \left[ - \frac{\partial_z s^2(z)}{(s^2(z))^2} +  \left[ 2(\mu - 2 + \frac{\kappa}{2})(s^1(z))^2 + 2 \partial_z s^1(z)\right] \right] \end{align*} where \begin{align*} & s^1(z) = - \frac{1}{2 z} \left(z^2-c+1 - \sqrt{(z^2-c+1)^2-4z^2} \right)\\ & s^2(z) = - \frac{1}{2 c z} \left(z^2+c-1 - \sqrt{(z^2-1+c)^2-4cz^2} \right) \end{align*}
\end{thm}

\begin{proof}
Recall that $M_{n,m}(z) = (n+m) (\mathbb{E}[s_{n,m}(z)] - s(z))$  can be written $$M_{n,m}(z)=n(\mathbb{E}[s_{n,m}^1(z)] - s^1(z)) + m(\mathbb{E}[s_{n,m}^2(z)] - s^2(z)) + o(1).$$

Recall that we have from \eqref{toto} that $$ \mathbb{E}[s_{n,m}^1(z)]  = -\frac{\left(a_1 - \sqrt{a_1^2 - 4 a_0 a_2}\right)}{2 a_2}$$
with \begin{align*} & a_2 = z + c \mathbb{E}[\delta_2(z)] + O(v_0^{-24}n^{-2})  \\ & a_1 = z^2 - c + 1 + (c-c_{n,m}) + z c \mathbb{E}[\delta_2(z)]  - z \mathbb{E}[\delta_1(z)]  + O(v_0^{-24}n^{-2}) \\ & a_0 = z - z^2\mathbb{E}[\delta_1(z)] +c \mathbb{E}[\delta_1(z)] + O(v_0^{-24}n^{-2}) \end{align*}  

By a Taylor expansion, we deduce
 \begin{align*} \sqrt{\left(\frac{a_1}{a_2}\right)^2 - 4 \frac{a_0}{a_2}} & = \sqrt{\frac{(z^2-c+1)^2}{z^2}-4} \\ & + \frac{1}{\sqrt{\frac{(z^2-c+1)^2}{z^2}-4}} \left[ \frac{(z^2-c+1)}{z} \frac{1}{z^2} \left( z (c-c_{n,m})-z^2 \mathbb{E}[\delta_1(z)]+ c(c-1)\mathbb{E}[\delta_2(z)] \right)\right] \\ & -  \frac{1}{\sqrt{\frac{(z^2-c+1)^2}{z^2}-4}} \left[ \frac{2c}{z}(\mathbb{E}[\delta_1(z)] - \mathbb{E}[\delta_2(z)]) - 2 z \mathbb{E}[\delta_1(z)] \right]  + O(v_0^{-25}n^{-2}). \end{align*}

Therefore, \begin{align*}\mathbb{E}[ s_{n,m}^1(z) & - s^1(z) ]  = - \frac{1}{2 z^2} \left( z (c-c_{n,m})-z^2 \mathbb{E}[\delta_1(z)]+ c(c-1)\mathbb{E}[\delta_2(z)] \right) \\ & + \frac{1}{2 \sqrt{\frac{(z^2-c+1)^2}{z^2}-4}} \left[ \frac{(z^2-c+1)}{z} \frac{1}{z^2} \left( z (c-c_{n,m})-z^2 \mathbb{E}[\delta_1(z)]+ c(c-1)\mathbb{E}[\delta_2(z)] \right)\right] \\ & -  \frac{1}{2 \sqrt{\frac{(z^2-c+1)^2}{z^2}-4}} \left[ \frac{2c}{z}(\mathbb{E}[\delta_1(z)] - \mathbb{E}[\delta_2(z)]) - 2 z \mathbb{E}[\delta_1(z)] \right]  + O(v_0^{-26}n^{-2}). \end{align*} 

For $\mathbb{E}[s_{n,m}^2(z)]$, we have similarly 
  \begin{align*} \mathbb{E}[s_{n,m}^2(z) & -s^2(z)]  = - \frac{1}{2 c^2 z^2} \left(- c^2 z^2 \mathbb{E}[\delta_2(z)] + (c_{n,m} - c)(z-z^3) - c(c-1) \mathbb{E}[\delta_1(z)]\right) \\ & + \frac{1}{2 \sqrt{\frac{(z^2+c-1)^2}{(cz)^2}-\frac{4}{c}}} \left[ \frac{(z^2+c-1)}{cz} \frac{1}{ - c^2 z^2} \left(- c^2 z^2 \mathbb{E}[\delta_2(z)] + (c_{n,m} - c)(z-z^3) - c(c-1) \mathbb{E}[\delta_1(z)]\right) \right] \\ & - \frac{1}{2 \sqrt{\frac{(z^2+c-1)^2}{(cz)^2}-\frac{4}{c}}} \left[ \frac{2}{c^2 z} \left( c (1 - z^2) \mathbb{E}[\delta_2(z)] - (c_{n,m}-c)z - c \mathbb{E}[\delta_1(z)]\right) \right]  + O(v_0^{-26}n^{-2}). \end{align*}

To find the limit of $M_{n,m}(z)$, we only need to find an equivalent of $c-c_{n,m}$, of $\mathbb{E}[\delta_1(z)]$ and of $\mathbb{E}[\delta_2(z)]$.
First, we have by Hypothesis \ref{hypo3}, $(c - c_{n,m}) = \frac{nc-m}{n} \sim \frac{\sigma}{n}.$
To compute an equivalent of $\mathbb{E}[\delta_2(z)]$, we will use the following formula $$\frac{1}{u - \epsilon } = \frac{1}{u} + \frac{\epsilon}{u^2} + \frac{\epsilon ^ 2 }{u^2 (u - \epsilon) }\,.$$
Applied to $\beta_k$, it gives 

$$\frac{1}{\beta_k} = \frac{1}{z+ \mathbb{E}[s_{n,m}^1(z)] - \epsilon_k} = \frac{1}{z+ \mathbb{E}[s_{n,m}^1(z)]} + \frac{\epsilon_k}{(z+ \mathbb{E}[s_{n,m}^1(z)])^2} + \frac{\epsilon_k ^2}{ \beta_k (z+ \mathbb{E}[s_{n,m}^1(z)])^2}.$$
Finally, we can write

\begin{align*}  \delta_2(z)&= - \frac{1}{m}\sum_{k=1}^{m} \frac{\epsilon_k}{(z+ \mathbb{E}[s_{n,m}^{1}(z)])^2} - \frac{1}{m}\sum_{k=1}^{m} \frac{\epsilon_k^2}{(z+ \mathbb{E}[s_{n,m}^{1}(z)])^3} - \frac{1}{m}\sum_{k=1}^{m} \frac{\epsilon_k^3}{ \beta_k (z+ \mathbb{E}[s_{n,m}^{1}(z)])^3} \\ & = S_1 + S_2 + S_3. \end{align*}
Let us find the limit of the expectation of 
each term.

First, since $z \in \mathbb{C}_{v_0}$, by the concentration Theorem \ref{concentrationthm}
applied to $A=S_k(z)$ we find for all $p\in\mathbb N$ a finite constant $C_p$ such that for all $k$
\begin{equation}\label{borneeps}\mathbb E[ |\epsilon_k|^p]\le \frac{C_p}{(\sqrt{n} v_0)^p}\,.\end{equation}
This implies that
 $$| \mathbb{E}[S_3] | \le v_0^{-4} \frac{1}{m} \sum_{k=1}^{m} \mathbb{E}[| \epsilon_k |^3] = O(\frac{v_0^{-7}}{n \sqrt{n}}).$$
Then, for $k \le m$, using a similar computation as done in the proof of Theorem \ref{THMMPlaw}, we have $$\mathbb{E}[\epsilon_k] = \frac{1}{n} \mathbb{E}\left[ \mathrm{Tr} (S_1 - (S_k)_1) \right].$$
By \cite[ Lemma 3.2 ]{EYY}, we have $$\mathrm{Tr} (S_1 - (S_k)_1) = \sum_{i=1}^{n} \frac{(W_{n,m} -z)^{-1}_{ki} (W_{n,m}-z)^{-1}_{ik}}{(W_{n,m}-z)^{-1}_{kk}}=\frac{ (W_{n,m} -z)^{-2}_{kk}}{ (W_{n,m}-z)^{-1}_{kk}} = \frac{((W_{n,m} -z)^{-1}_{kk})'}{(W_{n,m}-z)^{-1}_{kk}}$$  
Therefore, we deduce that
$$n\mathbb E[ \epsilon_k] =\mathbb E[\mathrm{Tr} (S_1 - (S_k)_1) ]=\mathbb E[ \frac{((W_{n,m}-z)^{-1}_{kk})'}{(W_{n,m}-z)^{-1}_{kk} }] $$
Theorem \ref{THMspeed} states that $(W_{n,m}-z)^{-1}_{kk}$ goes to ${s^2(z)}$ in $L_2$ provided $v_0\ge n^{-1/20}$.  Moreover,  we may assume without loss of generality that the eigenvalues 
of $W_{n,m}$ are bounded by some $\Lambda$ as by Lemma \ref{borne} for $\Lambda$ big enough
$$E[1_{\|W_{n,m}\|\ge \Lambda}\epsilon_k]\le 2 nv_0^{-1}e^{-\alpha \Lambda n}\,.$$
On $\|W_{n,m}\|\le \Lambda$, $(W_{n,m}-z)^{-1}_{kk}$ is lower bounded by a constant $C(\Lambda)>0$ and hence $$|n\mathbb E[ \epsilon_k] -\frac{1}{s^2(z)}\partial_z \mathbb E[(W_{n,m}-z)^{-1}_{kk}]|\le  2 nv_0^{-1}e^{-\alpha \Lambda n}
+C(\Lambda)^{-2} v_0^{-2} n^{-1} v_0^{-16}\,.$$
Finally $z\ra \mathbb E[(W_{n,m}-z)^{-1}_{kk}]$ is analytic, and bounded by Theorem \ref{THMspeed}  provided $v_0\ge n^{-1/16}$.  Hence,  writing Cauchy formula, we check that its derivative
converges towards the derivative of $s^2$  for $v_0\ge n^{-1/19}$.
Therefore, we conclude that uniformly on  $v_0\ge n^{-1/19}$, 
$$ n \mathbb{E}[\epsilon_k] - \frac{\partial_z s^2(z)}{s^2(z)}\to 0.$$
The last term to compute is $\mathbb{E}[\epsilon_k^2]$.
We  write $$\mathbb{E}[\epsilon_k^2] = \mathrm{Var}(\epsilon_k) + \mathbb{E}[\epsilon_k]^2 = \mathrm{Var}(\epsilon_k) + O(\frac{1}{n^2}).$$
We only need to compute the limit of the variance of $\epsilon_k$.

$$\mathrm{Var}(\epsilon_k) = \frac{1}{n^2} \mathrm{Var}( \alpha_k^* S_k \alpha_k ). $$
We decompose this variance as follows

\begin{equation} \label{decomp}\mathrm{Var} ( \alpha_k^* S_k \alpha_k ) =T_1^n+T_2^n+T^n_3+T^n_4+T^n_5+T^n_6\end{equation}
where if we denote in short  $\alpha_k = (X_i)_i$ and $(S_k)_1 = (s_{ij})_{ij}$ we have

\begin{eqnarray*}
T_1^n&=& \sum_{i,j} \mathrm{Var} (X_i s_{ij} X_j) \\
T_2^n&=&  \sum_{i,p} \mathrm{Cov}(s_{ii}X_i^2, s_{pp}X_p^2) \\ 
T_3^n&=&
 2 \sum_i \sum_{p \neq q}\mathrm{Cov}(X_i^2 s_{ii}, X_p s_{pq} X_q) \\ 
 T_4^n& =& 2 \sum_i \sum_{j \neq p} \mathrm{Cov}(X_i s_{ij} X_j, X_p s_{pi} X_i) \\ 
 T_5^n& =& \sum_{i,j} \mathrm{Cov}(X_i s_{ij} X_j, X_j s_{ji} X_i) \\ 
 T_6^n& =& \sum_{i \neq j \neq p \neq q} \mathrm{Cov} (X_i s_{ij} X_j, X_p s_{pq} X_q)\\
 \end{eqnarray*}
 We shall prove that $n^{-1}T^n_i$ converges for $i=1,2,5$ to a non zero limit whereas for $i=3,4,6$ it goes to zero.
Let us compute an equivalent for each term. On the way, we shall use the symmetry of the law
of $X$, which implies that the law of the matrix $W_{n,m}$ is invariant under $\left(W_{n,m}(ij)\right)\ra \left(W_{n,m}(\sigma(i)\sigma(j))\right)$ for any permutation $\sigma$ keeping fixed  $\{1,\ldots,m\}$ 
but permuting $\{m+1,\ldots,n+m\}$.  Therefore, we deduce  for instance that $\mathbb{E}[s_{ii} s_{pq}] =\mathbb{E}[s_{11}s_{23}] $ for all distinct $i,p,q<m$. 
We shall now estimate these terms by taking advantage of  some linear algebra tricks used already in e.g. \cite[Lemma 3.2]{EYY}.  

Let $\mathbb{T}$ be a set of $[\![1, (n+m) ]\!]$. Let us denote by $W_{n,m}^{(\mathbb{T})}$ the submatrix of $W_{n,m}$ obtained by deleting the $i$-th row and column, for $i \in \mathbb{T}$. To simplify the notations, let us denote by $(ij\mathbb{T})$ the set $(i \cup j \cup \mathbb{T})$.

Let us introduce the following notations : $$Z_{ij}^{(\mathbb{T})} = \frac{1}{n} \sum_{p,q} (\alpha_i)_p (W_{n,m}^{(\mathbb{T})}-z)^{-1}(p,q) (\alpha_j)_q,$$
$$K_{ij}^{(\mathbb{T})} = \frac{1}{\sqrt{n}} (X_j)_i - z \delta_{ij} - Z_{ij}^{(\mathbb{T})}.$$

We have the following formulas  for $i\neq j\neq k$ 
\begin{eqnarray}
s_{ij} &=& (W_{n,m}^{(k)}-z)^{-1}(i,j) =  - s_{jj} s_{ii}^{(j)} K_{ij}^{(ijk)}\label{la1} \\
s_{ii}&=&s_{ii}^j +s_{ij}s_{ji} s_{jj}^{-1}\label{la2}\\
s_{ij}&=&s_{ij}^k +s_{ik}s_{kj} s_{kk}^{-1}\label{la3}\end{eqnarray}

The concentration inequality states that, for $i \neq j$,  $\mathbb{E}[|Z_{ij}^{(\mathbb{T})}|^p] = O(v_0^{-p}n^{-\frac{p}{2}})$. Thus, 
\begin{equation}\label{b3}
\mathbb{E}[|K_{ij}^{(\mathbb{T})}|^p] = O(v_0^{-p}n^{-\frac{p}{2}})\,.\end{equation}
As a consequence, \eqref{la1} implies that for all $p$, for $i\neq j$
\begin{equation}\label{bo}
\mathbb E[|s_{ij}|^p]=O(v_0^{-3p}n^{-\frac{p}{2}})\,.\end{equation}
Moreover,  if $z=E+i\eta$, we find 
$$\Im s_{jj}(z)=\langle e_j, \frac{\eta}{\eta^2+(E-W)^2} e_i\rangle\ge \frac{\eta}{(2M)^2}$$
on $\|W\|\le M$, $\eta\le M$. Hence, we can use Lemma \ref{borne} to find that 
$$\mathbb E[ |s_{jj}^{-1}|^p]=O(\frac{1}{v_0^p})$$
Therefore, from \eqref{la2} and \eqref{bo}, we deduce
\begin{equation}\label{bo2}
\mathbb E[|s_{ii}-s_{ii}^{(j)} |^p]=O(v_0^{-3p}n^{-p}),\qquad \mathbb E[|s_{ij}-s_{ij}^{(k)}|^p]=O(v_0^{-3p}n^{-p})
\,.\end{equation}
We next show that for all $i\neq p\neq q$
\begin{equation}\label{b3}
\mathbb{E}[s_{ii} s_{pq}] =O(\frac{v_0^{-9}}{n \sqrt{n}})
\end{equation}
Let us first bound $$\mathbb E[s_{pq}]=- \mathbb E[ s_{pp}s_{qq}^{(p)} K^{pqk}_{pq}]= - \mathbb E[ s_{pp}^{)q)} s_{qq}^{(p)} K^{(pqk)}_{pq}] +O(\frac{v_0^{-12}}{n \sqrt{n}})$$
by \eqref{bo2} and \eqref{b3}. 
Denoting $e_p := s_{pp}^q - s^1(z)$ and $e_q:= s_{qq}^{(p)} - s^1(z)$, we have \begin{align*} \mathbb{E}[s_{pq}] & = \mathbb{E}[- s_{pp}^q s_{qq}^{(p)} K_{pq}^{(pqk)}] = - \mathbb{E}[(e_p +s^1(z))(e_q+s^1(z))K_{pq}^{(pqk)}] \\ & =  s^1(z) \mathbb{E}[e_p K_{pq}^{(pqk)}] + s^1(z) \mathbb{E}[e_q K_{pq}^{(pqk)}] + \mathbb{E}[e_p e_q K_{pq}^{(pqk)}]). \end{align*}
where we used that  $\mathbb{E} [K_{pq}^{(pqk)}] = 0$, because $p \neq q$. Moreover, recall that when we estimated $e_p$ in the proof of Theorem \ref{THMspeed}, we had
$$e_p=\frac{-s^1(z)+\frac{1}{n}\alpha_p^*s^{(qp)}\alpha_p}{(z+s^1(z))^2} +O(\frac{1}{v_0^3}|-s^1(z)+\frac{1}{n}\alpha_p^*s\alpha_p|^2)=\frac{\frac{1}{n}\alpha_p^*s^{(qp)}\alpha_p-\frac{1}{n}\Tr s^{(qp)} }{(z+s^1(z))^2} +O(v_0^{-19} n^{-1})$$
using the independence of $\alpha_p$ and $\alpha_q$, and their centering,  we see that for $r=p$ or $q$
$$\mathbb{E}[(\frac{1}{n}\alpha_p^*s^{(qp)}\alpha_p-\frac{1}{n}\Tr s^{(qp)}) K_{pq}^{(pqk)}] =0\,.$$
Hence we deduce that
\begin{equation}\label{b4}
\mathbb{E}[ s_{pq}] =O(\frac{v_0^{-20}}{n \sqrt{n}})\,.
\end{equation}
It remains to bound $$\mathbb{E}[ s_{pq}(s_{ii} -s^1(z))]=\mathbb{E}[ s_{pp}s_{qq}^{(p)}  K_{pq}^{(pqk)}(s_{ii} -s^1(z))]\,.$$
 Thanks to \eqref{bo2} and \eqref{bo}, we can replace $s_{pp}$ (resp. $s_{qq}^{(p)}$, resp. $s_{ii}$) by $s_{pp}^{(qi)}, s_{qq}^{(pi)}, s_{ii}^{(qp)}$).
 We can then apply the same argument as before as 
 $$\mathbb{E}[ s_{pp}^{(qi)}s_{qq}^{(pi)}  K_{pq}^{(pqki)}(\frac{1}{n}\alpha_i^*s^{(qpi)}\alpha_i-\frac{1}{n}\Tr s^{(qpi)})]=0$$
 to conclude that
 \begin{equation}\label{b3}
\mathbb{E}[(s_{ii} -s^1(z))s_{pq}] =O(\frac{v_0^{-20}}{n \sqrt{n}})\,.
\end{equation}

\begin{itemize}
\item {\it Estimates of $T^n_1$ and $T^n_5$.}
We first  prove that for  $v_0\ge n^{-1/20}$,
uniformly on $\mathbb C_{v_0}$, we have 
\begin{equation}\label{b1}\frac{1}{n}T^n_1=\frac{1}{n} \sum_{i,j} \mathrm{Var} (X_i s_{ij} X_j )\to (\mu - 2) (s^1(z))^2 + \partial_z s^1(z).\end{equation}
In fact, let us write the following decomposition:
$$\frac{1}{n}T^n_1=\frac{1}{n}\sum_i \mbox{Var}(X_i^2s_{ii}) +\frac{1}{n} \sum_{i\neq j}  \mbox{Var}(X_i X_j s_{ij}) =\frac{1}{n}T^n_{11}+\frac{1}{n}T^n_{12}\,.$$
To estimate the first term, observe that 
 $$| \mathbb{E}[s_{ii}^2 - (s^1(z))^2] |= |\mathbb{E}[(s_{ii} - s^1(z))(s_{ii}+s^1(z))]|\le \frac{2}{v_0}\frac{1}{\sqrt{v_0^16}n}$$
by Cauchy-Schwartz inequality and Theorem \ref{THMspeed}. Hence,
 $\mathbb{E}[s_{ii}^2] \to (s^1(z))^2$ for $v_0\ge n^{-1/19}$. Similarly, we can show that $\mathbb{E}[s_{ii}]^2 \to (s^1(z))^2$. Therefore we deduce that, since $\mu = \mathbb{E}[X_i^4]$,
 uniformly on $v_0\ge n^{-1/19}$, we have 
 \begin{equation}\label{lop}\frac{1}{n}T^n_{11}=\frac{1}{n} \sum (\mathbb E[X_i^4]\mathbb E[s_{ii}^2]-\mathbb E[s_{ii}]^2)\simeq  (\mu - 1) (s^1(z))^2.\end{equation}
Moreover,  
$$T^n_{12}= \sum_{i \neq j} \mathbb{E}[X_i^2 X_j^2] \mathbb{E}[s_{ij}^2] =  \sum_{i \neq j} \mathbb{E}[s_{ij}^2] +  \sum_{i \neq j} \mathbb{E}[s_{ij}^2]\mathbb E[ (X_i^2-1)(X_j^2-1)]=T^n_{121}+T^n_{122}$$
where we can estimate the second  term  $T^n_{122}$ by noticing that $ \mathbb{E}[s_{ij}^2]$ does not depend on $ij$ so that we can factorize it and deduce that
$$T^n_{122}= \sum_{i \neq j} \mathbb{E}[s_{ij}^2]\mathbb E[ (X_i^2-1)(X_j^2-1)]= \mathbb{E}[s_{m+1,m+2}^2]\left( \mathbb E[ (\sum_i(X_i^2-1))^2]-\mathbb E[\sum(X_i^2-1)^2]\right)$$
By \eqref{bo}, $ \mathbb{E}[s_{m+1,m+2}^2]$ is bounded by $v_0^{-6} n^{-1}$ whereas by we can bound the last term by  $Cn$ by 
concentration of measure, Theorem \ref{concentrationthm}, which yields: for all $\ell\in\mathbb N$ and $\delta>0$
\begin{equation}\label{concl}
\mathbb P( |\sum_p X_p|\ge \delta \sqrt{N})\le e^{-c_0\delta^2}\qquad \mathbb E( |\sum_p (X_p^2-1)|^p)\le C_p \sqrt{n}^p \end{equation}
This implies that $T^n_{122}$ is bounded by $C v_0^{-6}$. For the first term in $T^n_{12}$ we have

 \begin{eqnarray} T^n_{121}&=&\sum_{i \neq j} \mathbb{E}[s_{ij}^2] = \mathbb{E}[\sum_{i,j}s_{ij}^2 - \sum_{i} s_{ii}^2]  = \mathbb{E}[\mathrm{Tr} (S_k)_1^2 - \sum_{i} s_{ii}^2 ]  = \mathbb{E}[ \mathrm{Tr} (S_k)'_1 ] - \sum_{i} \mathbb{E}[s_{ii}^2]\nonumber \\ 
 & \sim & n (\partial_z s^1(z) - (s^1(z))^2)+ O(v_0^{-18}).\label{hjk}\ \end{eqnarray}
Therefore, we deduce \eqref{b1} from \eqref{lop} and \eqref{hjk}. 
To prove the convergence of $T^n_5$, notice that
 $$\sum_{i,j} \mathrm{Cov}(X_i s_{ij} X_j, X_j s_{ji} X_i) = \sum_{i,j} \mathrm{Var} (X_i s_{ij} X_j),$$ so that by the previous computation
$$ \frac{1}{n}T^n_5=\frac{1}{n} \sum_{i,j} \mathrm{Cov}(X_i s_{ij} X_j, X_j s_{ji} X_i) \to (\mu - 2) (s^1(z))^2 + \partial_z s^1(z).$$

\item  {\it Estimate of $T^n_2 $ .}

For the second  term, we have  by using that $\mathbb{E}[s_{ii} s_{pp}] $ is independent of $i\neq p$.
\begin{align*} T^n_2&=\sum_{i,p} \mathrm{Cov}(s_{ii}X_i^2, s_{pp}X_p^2) \\
& = \sum_{i,p}( \mathbb{E}[s_{ii}s_{pp}]\mathbb{E}[X_i^2 X_p^2] - \mathbb{E}[s_{ii}]\mathbb{E}[s_{pp}] )\\ & = \sum_{i,p} (\mathbb{E}[s_{ii}s_{pp}] - \mathbb{E}[s_{ii}]\mathbb{E}[s_{pp}]) + \sum_{i,p} (\mathbb{E}[X_i^2 X_p^2]-1) \mathbb{E}[s_{ii} s_{pp}] \\ & = \mathrm{Var}(\mathrm{Tr}(S_1)) + \mathrm{Var}(\sum_{i=1}^n X_i^2) (s^1(z))^2 + o(n). \end{align*}

Let $f$ be the function $f : x \mapsto (x-z)^{-1}$. $f$ is a Lipschitz function, whose Lipschitz constant is bounded by $v_0^{-2}$. By Corollary \ref{concspectral} (see the appendix \ref{concentrationsec}), we can neglect the variance of $\mathrm{Tr}(S_1)$. Since $\lim \frac{1}{n} \mathrm{Var}(\sum_{i=1}^n X_i^2) = \kappa$, we now have $$ \frac{1}{n} \sum_{i,p} \mathrm{Cov}(s_{ii}X_i^2, s_{pp}X_p^2) \to \kappa (s^1(z))^2.$$

\item  {\it Estimate of $T^n_3$}
Now, we will prove that the other terms can all be neglected and first  estimate $T^n_3$. We  start by the following expansion:

\begin{align*} \sum_i \sum_{p \neq q}\mathrm{Cov}(X_i^2 s_{ii}, X_p s_{pq} X_q) & = \sum_i \sum_{p \neq q}\mathbb{E}[X_i^2 X_p X_q] \mathbb{E}[s_{ii} s_{pq}] \\ & = \sum_i \sum_{p \neq q}\mathbb{E}[(X_i^2-1) X_p X_q] \mathbb{E}[s_{ii} s_{pq}] \end{align*}

We therefore have, by symmetry 
 of $\mathbb E[s_{ii}s_{pq}]$, that
  \begin{eqnarray}
T^n_3&=&
 \sum_i \sum_{p \neq q}\mathbb{E}[(X_i^2-1) X_p X_q] \mathbb{E}[s_{ii} s_{pq}] \nonumber\\
 &=& \mathbb{E}[s_{11}s_{23}] \mathbb E[\left(\sum_i (X_i^2-1)\right)\left(\sum_{p\neq q} X_pX_q\right)]\label{poi}\\
 &&+ \sum_{(pq)=(12),(21)}(\mathbb{E}[s_{11}s_{pq}] -\mathbb E[s_{11}s_{23}])(\mathbb E[\sum_i (X_i^2-1)X_i \sum_p X_p]-\mathbb E[\sum_i (X_i^2-1)X_i^2])\nonumber\end{eqnarray}
The last expectations are  at most of order $n\sqrt{n}$ by \eqref{concl} whereas \eqref{b3} shows that the expectations over the $s_{ij}$ are at most of order $v_0^{-12}/n\sqrt{n}$. 
Therefore,  $T_3^n$ is bounded by $O(v_0^{-14})$. 

\item {\it Estimate of $T^n_4$ and $T^n_6$}

For $i \neq j \neq p \neq q$, since the indices are bigger than $m$ (it is the matrix $(S_k)_1$ which is concerned),  
we have as for the estimation of $T^n_3$ 
 \begin{align*} \mathbb{E}[s_{ij} s_{pq}] & = O(\frac{1}{v_0^{14} n\sqrt{n}}). \end{align*}
Moreover, this term does not depend on the choices of $i\neq j \neq p \neq q$ and 
we also have by concentration of measure  $$\sum_{i \neq j \neq p \neq q} \mathbb{E}[X_i X_j X_p X_q] = O({n^2}).$$

Therefore,  $$\frac{T_6^n}{n}= \frac{1}{n} \sum_{i \neq j \neq p \neq q} \mathrm{Cov} (X_i s_{ij} X_j, X_p s_{pq} X_q) = O(\frac{1}{v_0^{14}\sqrt{n}}).$$

Finally, using the same reasoning, we show that $$\frac{T_4^n}{n}=  \frac{1}{n} \sum_i \sum_{j \neq p} \mathrm{Cov}(X_i s_{ij} X_j, X_p s_{pi} X_i)=O(\frac{1}{v_0^{18}\sqrt{n}}).$$

Indeed, if $ i \neq j \neq p$, we know that $\mathbb{E}[s_{ij}s_{pi}]= O(\frac{v_0^{-18}}{n \sqrt {n}})$. And, if $i =j$ or $i=p$, we have $\mathbb{E}[s_{ii}s_{pi}] = O(\frac{v_0^{-9}}{n})$. In both cases, $\mathbb{E}[s_{ij}s_{pi}]= O(\frac{v_0^{-18}}{n})$, which ends the proof.

To conclude, for $k \le m$, we have proved that uniformly on $v_0\ge n^{-1/36}$, $$ \lim_{(n,m) \to \infty, m/n \to c} m \mathbb{E}[\epsilon_k^2] = 2 c ( \mu -2 + \frac{\kappa}{2})(s^1(z))^2 + 2c \partial_z s^1(z).$$
This aim plies that 
 $$\mathbb{E}[\delta_2(z)] = - \frac{c}{m}(s^2(z))^2  \frac{\partial_z s^2(z)}{s^2(z)}  + \frac{c}{m} (s^2(z))^3 \left[ 2(\mu - 2 + \frac{\kappa}{2})(s^1(z))^2 + 2\partial_z s^1(z)\right].$$
For $k > m$, using the same method, we have uniformly on $v_0\ge n^{-1/36}$,
$$ n \mathbb{E}[\epsilon_k] \to \frac{\partial_z s^1(z)}{s^1(z)}.$$
This time, for $\mathbb{E}[\epsilon_k]$, $k > m$, the computations are almost the same as previously, except for the fact that the vectors are now independent. Therefore, we have $\mathrm{Var}(\sum_i X_i^2) = m( \mu - 1)$.
Thus, for $k > m$, we have the uniform convergence for $v_0\ge n^{-1/36}$;
$$ \lim_{(n,m) \to \infty, m/n \to c} n \mathbb{E}[\epsilon_k^2] =  c\left( (4\mu -5)(s^2(z))^2 + 2\partial_z s^2(z)\right)$$.

Summing these estimates, we deduce  that uniformly on $v_0\ge n^{-1/36}$,
$$\mathbb{E}[\delta_1(z)] \sim - \frac{1}{n} (s^1(z))^2 \frac{\partial_z s^1(z)}{s^1(z)} + \frac{c}{n} (s^1(z))^3 \left[ (4\mu -5)(s^2(z))^2 + 2\partial_z s^2(z)\right].$$

Noticing that $m = nc + O(1)$, we have  $$ n(\mathbb{E}[s_{n,m}^1(z)] - s^1(z)) +  m(\mathbb{E}[s_{n,m}^2(z)] - s^2(z)) = n[c_{n,m} - c] A(z) + n\mathbb{E}[\delta_1(z)] B(z) + m\mathbb{E}[\delta_2(z)] C(z) + O(n^{-1})$$ where \begin{align*} & A(z) = - \frac{(z^2+c-1)}{2zc} + \frac{1}{2\sqrt{\frac{(z^2-c+1)^2}{z^2} -4}} \left[ \frac{z^2-c+1}{z^2} - \frac{z^2+c-1}{cz^2}(z^2-1) + 2 \right]\\ & B(z) = \frac{1}{2} + \frac{c-1}{2z^2} + \frac{1}{2\sqrt{\frac{(z^2-c+1)^2}{z^2} -4}} \left[ z - \frac{(c-1)^2}{z^3} \right] \\& C(z) = \frac{1}{2} - \frac{c-1}{2z^2} + \frac{1}{2\sqrt{\frac{(z^2-c+1)^2}{z^2} -4}} \left[ z - \frac{(c-1)^2}{z^3} \right] \end{align*}

But, using the expressions of $s^1(z)$ and $s^2(z)$, we can see that $$ A(z) = s^2(z), \:  B(z) = 1 + c \partial_z s^2(z), \: C(z) = 1 + \partial_z s^1(z).$$

Therefore, if we define $M(z)$ by \begin{align*} M(z) := & \sigma s^2(z) + (1 + c \partial_z s^2(z)) (s^1(z))^3 \left[  - \frac{\partial_z s^1(z)}{ (s^1(z))^2} + c \left[ (4\mu -5)(s^2(z))^2 + 2\partial_z s^2(z)\right] \right] \\ & +  (1 + \partial_z s^1(z)) c (s^2(z))^3 \left[ - \frac{\partial_z s^2(z)}{(s^2(z))^2} +  \left[ 2(\mu - 2 + \frac{\kappa}{2})(s^1(z))^2 + 2 \partial_z s^1(z)\right] \right] \end{align*} we have proved
that uniformly on $v_0\ge n^{-1/36}$,
$$M_{n,m}(z) \to  M(z).$$ 

\end{itemize}
\end{proof}

\section{Appendix}

\subsection{Concentration of measure}\label{concentrationsec}

\begin{thm}\label{concentrationthm}
Let $X$ be a random vector whose distribution satisfies Hypothesis \ref{hypo1} and \eqref{hypo2}. For all symmetric matrix $A$, for all integer $p$, there exist $C_p < \infty$ such that $$\mathbb{E}[{|X^* A X - \mathrm{Tr}(A)|}^p] \le C_p 	{\| A \|}^p {\sqrt{n}}^p  \: \: \qquad \| A \| = \sup_{\| y \| = 1} \| Ay \|, \: \: \: \|.\| \: \text{ euclidean norm}.$$
\end{thm}

\begin{proof}  
 We decompose the covariance as :

\begin{equation}\label{cov}\mathbb{E}[{|X^* A X - \mathrm{Tr}(A)|}^p] \le 2^{p-1} (\mathbb{E}[|\sum_{i=1}^N (X_i^2 - \mathbb{E}[X_i^2])a_{ii}|^p] + \mathbb{E}[{|\sum_{i \neq j} X_i X_j a_{ij}|}^p]).\end{equation}

We can now find an upper bound for each term. 
For this, we will use a remarkable property of $V$. Since Hess $V$ $ \ge \frac{1}{C} \: I_n$ , the logarithmic Sobolev inequality holds for the measure $dX$ with the constant C (see e.g. \cite[ Lemma 2.3.2 and 2.3.3]{AGZ}). Consequently, for all Lipschitz function $f : \mathbb{R}^n \rightarrow \mathbb{R}$, if we write $|f|_{\mathcal{L}}$ its Lipschitz constant, we have $$\mathbb{P}(|f(X)- \mathbb{E}[f(X)]| \ge t) \le 2 e^{\frac{-t^2}{2c |f|^2_{\mathcal{L}}}}.$$
Let us focus on the first term on the right hand side of \eqref{cov}.

First, let us suppose that for all $i$, $a_{ii} \ge 0$. Let $f : x \mapsto \sqrt{\sum a_{ii} x_i^2}$, we will show that this is a Lipschitz function whose constant is bounded by $\|A\|_{\infty}^{\frac{1}{2}}$. Indeed, $\frac{\partial f}{ \partial x_i} = a_{ii} \frac{x_i}{f(x)}$. Therefore, since $a_{ii} \ge 0$, we have $\| \nabla f \|^2 \le \max |a_{ii}| \le \|A\|_{\infty} $. Taylor's inequality states that $|f|_{\mathcal{L}} \le \|A\|_{\infty}^{\frac{1}{2}} $. For all $p \in \mathbb{N}$, we now have : \\ \begin{align*} \mathbb{E}[{|\sqrt{\sum a_{ii} X_i^2} - \mathbb{E}[\sqrt{\sum a_{ii} X_i^2}]|}^p] &= \int_{0}^{\infty} t^{p-1} \mathbb{P}(|\sqrt{\sum a_{ii} X_i^2} - \mathbb{E}[\sqrt{\sum a_{ii} X_i^2}]| \ge t) dt \\ \ & \le 2 \int_{0}^{\infty} t^{p-1} e^{\frac{-t^2}{2 \|A\|_{\infty} C}} dt   \le C_p \|A\|_{\infty}^{\frac{p}{2}}. \end{align*}

Thus, we have : 
$$
C_1\le 2^{p-1}  (\mathbb{E}[|{\sum a_{ii} X_i^2} - \mathbb{E}[\sqrt{\sum a_{ii} X_i^2}]^2|^p] +  \mathbb{E}[|\mathbb{E}[\sum a_{ii} X_i^2] - \mathbb{E}[\sqrt{\sum a_{ii} X_i^2}]^2|^p]),$$
By the concentration inequality, we know that the second  term is bounded regardless of $n$. Then, for the first term, Cauchy-Schwartz inequality gives 
\begin{eqnarray*}
\mathbb{E}[{|{\sum a_{ii}X_i^2} - \mathbb{E}[\sqrt{\sum a_{ii} X_i^2}]^2|}^p] &=& \mathbb{E}[|\sqrt{\sum a_{ii} X_i^2} - \mathbb{E}[\sqrt{\sum a_{ii} X_i^2}]|^p |\sqrt{\sum a_{ii} X_i^2} + \mathbb{E}[\sqrt{\sum a_{ii} X_i^2}]|^p],\\
&\le& \mathbb{E}[|\sqrt{\sum a_{ii} X_i^2} - \mathbb{E}[\sqrt{\sum a_{ii} X_i^2}]|^{2p}]^{\frac{1}{2}} \mathbb{E}[|\sqrt{\sum a_{ii} X_i^2} + \mathbb{E}[\sqrt{\sum a_{ii} X_i^2}]|^{2p}]^{\frac{1}{2}}.\end{eqnarray*}
Using again the concentration inequality, we show that the first term of the product can be bounded by the product of $\|A\|_{\infty}^p$ and a constant depending only on $p$. Assembling all these terms alongside with noticing that $\|A\|_{\infty} \le \|A\|$, we have proved that there exist $C_p$ independent of $n$ such that : $$\mathbb{E}[|\sum_{i=1}^n ( X_i^2 - \mathbb{E}[X_i^2])a_{ii}|^p] \le C_p \|A\|^p \sqrt{n}^p.$$
Now, for the general case, let us write $A$ as a sum of a positive and a negative matrix $A = A^+ - A^-$. Therefore,  \begin{align*} \mathbb{E}[|\sum_{i=1}^N (X_i^2 - \mathbb{E}[X_i^2])a_{ii}|^p] & = \mathbb{E}[|\sum_{i=1}^N (X_i^2 - \mathbb{E}[X_i^2])a_{ii}^+ - a_{ii}^-|^p] \\ & \le 2^{p-1} (\mathbb{E}[|\sum_{i=1}^N (X_i^2 - \mathbb{E}[X_i^2])a_{ii}^+|^p] +  \mathbb{E}[|\sum_{i=1}^N (X_i^2 - \mathbb{E}[X_i^2])a_{ii}^-|^p]). \end{align*}
Since the $a_{ii}^+$ and the $a_{ii}^-$ are all non negative, the previous proof still holds and therefore, we have $$\mathbb{E}[|\sum_{i=1}^n ( X_i^2 - \mathbb{E}[X_i^2])a_{ii}|^p] \le C_p \|A\|^p \sqrt{n}^p.$$
On the other hand,  if we denote by $Z_n$ the cardinal number of partitions of $[\![1,n ]\!]$ into two distincts sets $I$ and $J$, we have $$\mathbb{E}[{|\sum_{i \neq j} X_i X_j a_{ij}|}^p] \le \frac{1}{Z_n} \sum_{I \cap J = \emptyset} \mathbb{E}[{|\sum_{(i,j) \in (I \times J)} X_i X_j a_{ij}|}^p].$$
We next expand the right hand side as
$$\mathbb{E}[{|\sum_{(i,j) \in (I \times J)} X_i X_j a_{ij}|}^p]=\sum_{i_1,\ldots,i_p\in I}\sum_{j_1,\ldots,j_p\in I}\prod_{k=1}^p a_{i_k,j_k} \mathbb E[X_{i_1}\cdots X_{i_p} X_{j_1}\cdots X_{j_p}]$$
In the above right hand side, assume that there are $K$ indices $s_1,\ldots,s_K$ with multiplicity one. Let $\ell_1,\ldots, \ell_r$ be the others.
 We then can apply the symmetry of the law of the $X_i$'s to see that for some $n_1,\ldots, n_r\ge 2$, 
$$\mathbb E[X_{i_1}\cdots X_{i_p} X_{j_1}\cdots X_{j_p}]=\mathbb E[X_{\ell_1}^{n_1}\ldots X_{\ell_r}^{n_r}
\prod_{k=1}^K  X_{s_k}]=\mathbb E[X_{1}^{n_1}\ldots X_{\ell_r}^{n_r}
\prod_{k=1}^K  \frac{1}{N_K}\sum_{s\in I_k} X_{s_k}]$$
where $I_k=[ r+1+(k-1)N_K,  r+kN_K]$ and $KN_K +\sum n_i =n$. 
Now, by concentration inequality applied to $\sum_{s\in I_k} X_{s_k}$ we deduce that there exists a finite constant $C_p$ such that
$$|\mathbb E[X_{i_1}\cdots X_{i_p} X_{j_1}\cdots X_{j_p}]|\le  \sqrt{n}^{-K}\,.$$
Using that $a_{ij}$ is bounded by $\|A\|$, we conclude that
$$\mathbb{E}[{|\sum_{(i,j) \in (I \times J)} X_i X_j a_{ij}|}^p]\le C_p\sum_{K=0}^p \|A\|^p n^{K/2} n^{\frac{p-K}{2}}\le  C_p\|A\|^p n^p$$
which completes the proof.

\end{proof}

\begin{remark}
Let $ 1 \le k \le n$. The logarithmic Sobolev inequality holds for each $X_i(k), i \in [\![1,m]\!]$ with the constant $C$. By independence of the $X_i(k), i \in [\![1,m]\!]$, it also holds for the vector $(X_1(k), \cdots, X_m(k))$ with the same constant. Therefore, the concentration inequality holds the same way for the vector $(X_1(k), \cdots, X_m(k))$ for all $k, 1 \le k \le n$.
\end{remark}
Moreover, the spectral measure satisfies concentration inequalities \cite[Theorem 2.3.5]{AGZ}:

\begin{corollaire}\label{concspectral} Under Hypothesis \ref{hypo1}, for all Lipschitz function $f$ on $\mathbb{R}$, for all $p \in \mathbb{N}$, for all $(n,m) \in \mathbb{N}^2$, there exist $C_p$ independent of $n$ and $m$ such that $$\mathbb{E}[|\mathrm{Tr}(f(W_{n,m}))-\mathbb{E}[\mathrm{Tr}(f(W_{n,m}))|^p] \le C_p |f|_{\mathcal{L}}^p$$.
 
\end{corollaire}

The final result we will need is the concentration of the extremal eigenvalues, see \cite[Corollary A.6]{AGZ}
\begin{thm}\label{conclargest} 
Let $\lambda_i$ be the ordered eigenvalues of $W_{n,m}$.
Then there exists a positive constant $c_0$ and a finite constant $C$ such that for all $N$
$$\mathbb{P}(|\lambda_i-\mathbb E[\lambda_i]|\ge \delta)\le C \exp\{-c_0 N\delta^2 \}$$
\end{thm}

\subsection{Boundedness of the spectrum of $W_{n,m}$}

We will now show that, asymptotically, the spectrum of $W_{n,m}$ is bounded. For this, we will compare the law of the matrix with the one of a Wigner matrix whose entries are all i.i.d centered Gaussian random variables. 

Indeed, 
\begin{lemma}\label{borne}
 Under Hypothesis \ref{hypo1},
 there exist $\alpha > 0$ and $C < \infty$ such that, for all integer $n$, we have $$\mathbb{P}(\lambda_{\max}(W_{n,m}) > C) \: \le \: e^{- \alpha C n}.$$
\end{lemma}

\begin{proof}

Since $\: \: W_{n,m}$ is symmetric, the law of $W_{n,m}$ is the law of $Y_{n,m}$ which we can write as :
$$dY_{n,m} \: = \: \frac{1}{{Z_n}} \prod_{j=1}^{m} e^{-V(\sqrt{n}(Y_{1j}, \dotsc, Y_{nj}))} \prod_{i}^{n} \prod_{j}^{m} dY_{i,j}.$$
By \eqref{hess} this law has a strictly log-concave density and therefore, 
if we denote  by $\gamma$ the Gaussian law
 $$d\gamma = \frac{1}{Z_C} \prod_{i}^{n} \prod_{j}^{m} e^{ - \frac{n}{2C} Y_{ij}^2} dY_{i,j}\,,$$
 we can apply Brascamp-Lieb inequality (\cite[Thm 6.17]{G}) which implies 
 that for all convex function $g$, we have : $$ \int g(x)\frac{f(x)d\gamma(x)}{\int f d\gamma} \le \int g(x) d\gamma (x).$$
Applying this inequality with $g(x)=e^{ns\lambda_{\max}(W_{n,m})}$ for $s>0$
we deduce that
$$ \int e^{sn \lambda_{\max}(W_{n,m})} dY_{n,m} \le \int e^{sn \lambda_{\max}(W_{n,m})} d\gamma.$$
The right hand side is bounded by $C^n$ for some finite constant $C$, see e.g. \cite[Section 2.6.2]{AGZ}. 
Tchebychev's inequality  completes  the proof.

\end{proof}
\subsection{An example for the function V}\label{subexample}

We will show in this part that there exists $V$ such that the hypotheses made in the introduction can be verified without jeopardizing the dependance of the random variables, i.e. there exists $V$ such that $\kappa$ is different from $\mu - 1$ (which is the result expected for the case where the random variables are independent).

Let $V$ be given by 

$$ V : (X_1,\dotsc, X_n) \mapsto \frac{a}{n}(\sum_{i=1}^n (X_i^2 -m_{b,c}))^2 + b \sum_{i=1}^n X_i^2 + c\sum_{i=1}^n X_i^4.$$

Therefore, the law of the random vector $X$ can be written as $$dX = \frac{1}{Z^n_{a,b,c}} e^{\left(-\frac{a}{n}(\sum_{i=1}^n (X_i^2 -m_{b,c}))^2 - b \sum_{i=1}^n X_i^2- c\sum_{i=1}^n X_i^4\right)} \prod_{i=1}^n dX_i$$

where, if we denote by $\mathbb{E}_{a,b,c}[.]$ the expectation under this measure, $m_{b,c} =\mathbb{E}_{0,b,c} [X_i^2]$. This law has obviously a strictly log-concave density
when $a,b,c$ are positive, and it is symmetric. Hence, Hypothesis \ref{hypo1} is fulfilled.

First, let us note that for all $a \ge 0$, we have $$\frac{1}{n}\sum_{i=1}^n X_i^2 \to m_{b,c}.$$

Now, let us estimate $Z^n_{a,b,c}$. For this, we will introduce a a Gaussian distribution $\mathcal{N}(0,1)$ as it follows : 
$$Z^n_{a,b,c} = Z^n_{0,b,c} \frac{1}{\sqrt{2\pi}}\int e^{(-\frac{g^2}{2})} \left(\mathbb{E}_{0,b,c}\left[e^{\left(\frac{i\sqrt{a}g}{\sqrt{n}} (X_i^2-m_{b,c})\right)}\right]\right)^n dg.$$

Thus, we have $$ Z^n_{a,b,c}\simeq Z^n_{0,b,c} \frac{1}{\sqrt{2\pi}}\int e^{-\frac{g^2}{2}} e^{-\frac{a^2}{2} g^2 \sigma_{b,c} +O(1/\sqrt n)}  dg,$$

where $$ \sigma_{b,c}=\mathbb{E}_{0,b,c}\left[ (X_i^2-m_{b,c})^2\right]\,.$$

Therefore, $$\log (Z^n_{a,b,c}) = \log (Z^n_{0,b,c}) - \frac{1}{2} \log (1+a^2\sigma_{b,c})+ O(1).$$

On the other side, we have \begin{eqnarray*}
\mathbb E_{a,b,c}[X_iX_j] &=&1_{i=j} \mathbb E_{a,b,c}[X_i^2]\\
\mathbb E_{a,b,c}[X_i^2]&=&-\frac{1}{n}\partial_b \log (Z^n_{a,b,c})=\mathbb E_{0,b,c}[X_i^2]+O(\frac{1}{n} )\\
\mathbb E_{a,b,c}[X_i^4]&=&-\frac{1}{n}\partial_c \log (Z^n_{a,b,c})=\mathbb E_{0,b,c}[X_i^2]+O(\frac{1}{n} )\\
\frac{1}{n}\mathbb E_{a,b,c} [(\sum (X_i^2-\mathbb E_{a,b,c}[X_i^2]))^2]&=&\frac{1}{n}\mathbb E_{a,b,c} [(\sum (X_i^2-m_{b,c} ))^2]- n(m_{b,c}-\mathbb E_{a,b,c}[X_i^2])^2\\
&\sim&
-\frac{1}{n}\partial_a \log (Z^n_{a,b,c})+O(\frac{1}{n})\\
\end{eqnarray*}

Summarizing what we have done, we can set the value of $\kappa = \frac{1}{n}\mathrm{Var}(\sum X_i^2)$ with the parameter $a$, regardless of the value of $\mu = \mathbb{E}[X_i^4]$. 
It is not hard to see that the symmetry condition of  Hypothesis \ref{hyp2} is fulfilled.

\subsection{Linear Algebra}
In this section we remind a few classical linear algebra identities.
\begin{lemma}\label{linalg}
Let $A$ be a square invertible matrix, and if $A$ is a block matrix, its determinant can be computed using the following formula :

$$\det{\begin{pmatrix} A & B \\ C & D \end{pmatrix}} = \det(A) \det(D - CA^{-1}B).$$
\end{lemma}

\begin{lemma}\label{Schur}
For $n \times n$ Hermitian nonsingular matrix $A$, define $A_k$, which is a matrix of size $(n-1)$, to be the matrix resulting from deleting the  $k$-th row and column of $A$. If both $A$ and $A_k$ are invertible, denoting $A^{-1} = (a^{ij})$ and $\alpha_k$ the vector obtained from the $k$-th column of $A$ by deleting the $k$-th entry,  we have

$$a^{kk} = \frac{1}{a_{kk} - \alpha_k^* A_k^{-1} \alpha_k}.$$

More precisely, if for all $k$, $A_k$ is invertible, we have $$\mathrm{Tr}(A^{-1}) = \sum_{i=1}^n \frac{1}{a_{kk} - \alpha_k^* A_k^{-1} \alpha_k}.$$
\end{lemma}

\begin{lemma}\label{linalg2}
If the matrix $A$ and $A_k$ are both nonsingular and hermitian, we have

$$\mathrm{Tr}(A^{-1}) - \mathrm{Tr}(A_k^{-1}) = \frac{1 + \alpha_k^* A_k^{-2} \alpha_k}{a_{kk} - \alpha_k^* A_k^{-1} \alpha_k}.$$
\end{lemma}

\section*{}


\begin{thebibliography}{5}


\bibitem{AGZ}
Greg W.Anderson, Alice Guionnet, Ofer Zeitouni,
\emph{An Introduction to Random Matrices}.
Cambridge studies in advanced mathematics,
2010.



 \bibitem{BAI2009EJP} Z.D.~Bai,  X. Wang, W. Zhou  \emph{CLT for linear spectral statistics of Wigner matrices.} Electron. J. Probab. {\bf 14} (2009), no. 83, 2391--2417.

\bibitem{BS0} Zhidong Bai, and Jack Silverstein, 
     \emph{No eigenvalues outside the support of the limiting spectral
              distribution of large-dimensional sample covariance matrices},
   {Ann. Probab.},
   {\bf 26}, {1998},
    {1},
    {316--345}.
\bibitem{BS}
Zhidong Bai, Jack W. Silverstein,
\emph{Spectral Analysis of Large Dimensional Random Matrices},  Springer, New York, 
Second Edition,
2010.
\bibitem{BY} Zhidong Bai, Z. D. and Y.Q Yin, 
     \emph{Limit of the smallest eigenvalue of a large-dimensional sample
              covariance matrix},
    {Ann. Probab.},
  {\bf 21},
     {1993}, 1275--1294.

 \bibitem{BaiYaoBernoulli2005} Z.D.~Bai, J. Yao  \emph{On the convergence of the spectral empirical process of Wigner matrices.} Bernoulli {\bf 11} (2005) 1059--1092.
\bibitem{chatterjee} S.~Chatterjee
     \emph{Fluctuations of eigenvalues and second order {P}oincar\'e
              inequalities},
  {Probab. Theory Related Fields},
 {\bf 143},
     {2009},{1--40}.
\bibitem{EYY}
Laszlo Erdos, Horng-Tzer Yau, Jun Yin
\emph{Rigidity of Eigenvalues of Generalized Wigner Matrices},
  {Adv. Math.}, {\bf{229}},
      {2012},
     {3},
      {1435--1515}.	
      \bibitem{GLPP}O. Gu\'edon, A. Lytova, A. Pajor, L. Pastur, \emph{The Central Limit Theorem for Linear Eigenvalue Statistics of the Sum of Independent Matrices of Rank One}
      arXiv: arXiv:1310.2506
\bibitem{G}
Alice Guionnet,
\emph{Large Random Matrices : Lectures on Macroscopic Asymptotics}.
Lecture Notes in Mathematics,
2009.

\bibitem{GU}
Alice Guionnet,
\emph{Grandes matrices al?atoires et th?or?mes d'universalit? (d'apr\`es {E}rd{\H o}s, {S}chlein, {T}ao,
              {V}u et {Y}au))}
      {S{\'e}minaire Bourbaki. Vol. 2009/2010. Expos{\'e}s
              1012--1026},
   {Ast\'erisque},
 {\bf 339},
      {2011}.
\bibitem{HT} Uffe Haagerup, and  Steen Thorbj{\o}rnsen, 
\textit{A new application of random matrices: {${\rm Ext}(C\sp *\sb
              {\rm red}(F\sb 2))$} is not a group},
  {Ann. of Math. (2)}
  {\bf 162}, 711--775, (2005).
  \bibitem{KKP96}  A. M. Khorunzhy, B. A. Khoruzhenko, L. A. Pastur \emph{Asymptotic properties of large random
matrices with independent entries}, {J. Math. Phys.} 37 (1996) 5033--5060.  
   \bibitem{johansson88} K.~Johansson \emph{On Szeg\"{o} asymptotic formula for Toeplitz determinants
and generalizations},
Bull. des Sciences Math\'ematiques, vol. 112 (1988), 257-304.
  
 \bibitem{johansson} K.~Johansson \emph{On the fluctuations of eigenvalues of random Hermitian matrices.} Duke Math. J. {\bf 91} 1998, 151--204.
 
    \bibitem{jonsson}  D.~Jonsson \emph{Some limit theorems for the eigenvalues of a sample covariance matrix.} J. Mult. Anal. {\bf 12}, 1982, 1--38.
    

    		\bibitem{MP} V. Marchenko and Leonid Pastur, \emph{the eigenvalue distribution in some ensembles of random matrices}, Math. USSR. Sbornik {\bf 1} , 457--483 1967
\bibitem{lytova}
A.~Lytova,    L.~ Pastur
     \emph{Central limit theorem for linear eigenvalue statistics of
              random matrices with independent entries},
    {Ann. Probab.}, {\bf 37}, {2009},
   1778--1840.

\bibitem{PP}
A.Pajor, L.Pastur
\emph{On the Limiting Empirical Measure of the sum of rank one matrices with log-concave distribution},
{Studia Math.},
 {\bf 195}, ( {2009})  {11--29}.
  \bibitem{MShcherbina11} M.~Shcherbina
\emph{Central Limit Theorem for Linear Eigenvalue Statistics of the Wigner and
Sample Covariance Random Matrices}, Journal of Mathematical Physics, Analysis,
Geometry, 7(2), (2011), 176--192.
\bibitem{sinai} Y.~Sinai,   A.~Soshnikov \emph{
     Central limit theorem for traces of large random symmetric
              matrices with independent matrix elements},
 {Bol. Soc. Brasil. Mat. (N.S.)}, {\bf 29}, {1998},
  1--24.

  \bibitem{soshni00} A. Soshnikov \emph{The central limit theorem for local linear statistics in
classical compact groups and related combinatorial identities}, Ann. Probab., 28 (2000), 1353--1370.
  


\bibitem{T}
Terence Tao,
\emph{Topics in random matrix theory},
{Graduate Studies in Mathematics},
  {\bf 132},
 {American Mathematical Society},
  {Providence, RI},
      {2012}.

\end{thebibliography}
\end{document}